\documentclass[12pt]{article}

\usepackage{amsmath,amsthm,amsfonts,txfonts,amssymb,amscd,epsf,epsfig,psfrag,enumerate,relsize}
\usepackage[mathscr]{eucal}
\usepackage{yhmath}
\usepackage{graphicx,epstopdf}
\usepackage{float}
\usepackage{wasysym}

\title{A Seifert algorithm for lamination links }

\author{Ulrich Oertel}

\date{Version 2, July, 2018}

\newtheorem{thm}{Theorem}[section] \newtheorem{lemma}[thm]{Lemma}

\newtheorem{proposition}[thm]{Proposition}
 
\newtheorem{questions}[thm]{Questions} 
\newtheorem*{claim*}{Claim}

 \theoremstyle{definition}
\newtheorem{defn}[thm]{Definition}
 
 \newtheorem{ex}[thm]{Example}

\theoremstyle{remark}

\begin{document}

\maketitle


\def\HDS{half-disk sum}

\def\PLength{\text{Length}}

\def\Area{\text{Area}}
\def\Im{\text{Im}}
\def\im{\text{Im}}
\def\cl{\text{cl}}
\def\rel{\text{ rel }}
\def\irred{irreducible}
\def\half{spinal pair }
\def\spinal{\half}
\def\spinals{\halfs}
\def\halfs{spinal pairs }
\def\reals{\mathbb R}
\def\rationals{\mathbb Q}
\def\complex{\mathbb C}
\def\naturals{\mathbb N}
\def\integers{\mathbb Z}
\def\id{\text{id}}
\def\Chi{\raise1.5pt \hbox{$\chi$}}
\def\cr{\tt\large}

\def\proj{P}
\def\hyp {\hbox {\rm {H \kern -2.8ex I}\kern 1.15ex}}

\def\Diff{\text{Diff}}

\def\weight#1#2#3{{#1}\raise2.5pt\hbox{$\centerdot$}\left({#2},{#3}\right)}
\def\intr{{\rm int}}
\def\inter{\ \raise4pt\hbox{$^\circ$}\kern -1.6ex}
\def\Cal{\cal}
\def\from{:}
\def\inverse{^{-1}}
\def\Max{{\rm Max}}
\def\Min{{\rm Min}}
\def\fr{{\rm fr}}
\def\embed{\hookrightarrow}
\def\Genus{{\rm Genus}}
\def\Z{Z}
\def\X{X}

\def\roster{\begin{enumerate}}
\def\endroster{\end{enumerate}}
\def\intersect{\cap}
\def\definition{\begin{defn}}
\def\enddefinition{\end{defn}}
\def\subhead{\subsection\{}
\def\theorem{thm}
\def\endsubhead{\}}
\def\head{\section\{}
\def\endhead{\}}
\def\example{\begin{ex}}
\def\endexample{\end{ex}}
\def\ves{\vs}
\def\mZ{{\mathbb Z}}
\def\M{M(\Phi)}
\def\bdry{\partial}
\def\hop{\vskip 0.15in}
\def\hip{\vskip0.1in}
\def\mathring{\inter}
\def\trip{\vskip 0.09in}
\def\PML{\mathscr{PML}}
\def\J{\mathscr{J}}
\def\G{\mathscr{G}}
\def\H{\mathscr{H}}
\def\C{\mathscr{C}}
\def\S{\mathscr{S}}
\def\S{\mathscr{S}}
\def\CT{\mathscr{CT}}
\def\WS{\mathscr{WS}}
\def\PS{\mathscr{PS}}
\def\I{\mathscr{I}}
\def\PI{\mathscr{PI}}
\def\T{\mathscr{T}}
\def\PT{\mathscr{PT}}
\def\WT{\mathscr{WT}}
\def\PWT{\mathscr{PWT}}
\def\E{\mathscr{E}}
\def\K{\mathscr{K}}
\def\L{\mathscr{L}}
\def\PC{\mathscr{PC}}
\def\PWC{\mathscr{PWC}}
\def\W{\mathscr{W}}
\def\WC{\mathscr{WC}}
\def\PWC{\mathscr{PWC}}
\def\RW{\mathscr{RW}}
\def\RC{\mathscr{RC}}
\def\PL{\mathscr{PL}}
\def\PLI{\mathscr{PL}{\it i}}
\def\PBLI{\mathscr{PBL}{\it i}}
\def\LI{\mathscr{L}{\it i}}
\def\CB{\mathscr{CB}}
\def\B{\mathscr{B}}
\def\PCB{\mathscr{PCB}}
\def\PT{\mathscr{PT}}
\def\W{\mathscr{W}}
\def\MC{\mathscr{MC}}

\def\PMF{\mathscr{PMF}}
\def\OO{\mathscr{O}}
\def\OM{\mathscr{OM}}
\def\POM{\mathscr{POM}}
\def\IM{\mathscr{IM}}
\def\PIM{\mathscr{PIM}}
\def\BIM{\mathscr{BIM}}
\def\PBIM{\mathscr{PBIM}}
\def\OMR{\mathscr{OMR}}
\def\POMR{\mathscr{POMR}}
\def\Preals{\mathscr{P}\mathbb R}
\def\PAA{\P\hskip -1mm\left[\reals^\A\times\reals^\A\right]}
\def\P{\mathscr{P}}
\def\suchthat{|}
\newcommand{\bigins}{\mathop{\mathlarger{\mathlarger\circledwedge}}}
\newcommand\invlimit{\varprojlim}
\newcommand\congruent{\equiv}
\newcommand\modulo[1]{\pmod{#1}}
\def\ML{\mathscr{ML}}
\def\Stack{\mathscr{T}}
\def\M{\mathscr{M}}
\def\A{\mathscr{A}}
\def\R{\mathscr{R}}
\def\union{\cup}
\def\atlas{\mathscr{A}}
\def\Int{\text{Int}}
\def\frontier{\text{Fr}}
\def\composed{\circ}

\def\abs{\odot}
\def\DS{\breve S}
\def\DL{\breve L}
\def\DBL{\breve{\bar L}}
\def\II{[0,\infty]}
\def\equiv{\hskip -3pt \sim}
\def\chim{\chi_-}

\def\split{\prec}
\def\pinch{\succ}
\def\OB{\mathbb O}
\def\FB{\mathbb F}
\def\SB{\mathbb S}

\def\TB{\mathbb T}
\def\OBB{{\mathbb S}\kern -6pt\raisebox{1.3pt}{--} \kern 2pt}
\def\Infty{\hbox{$\infty$\kern -8.1pt\raisebox{0.2pt}{--}\kern 1pt}}
\def\ens{\bar\circledwedge}
\def\ins{\circledwedge}
\def\rins{\circledvee}
\def\rens{\bar\circledvee}
\def\isom{\cong}
\def\bov{{\bf v}}
\def\boa{{\bf a}}
\def\boc{{\bf c}}
\def\bob{{\bf b}}
\def\bow{{\bf w}}
\def\bou{{\bf u}}
\def\boy{{\bf y}}
\def\bor{{\bf r}}
\def\bot{{\bf t}}
\def\boq{{\bf q}}
\def\boz{{\bf z}}
\def\boxx{{\bf x}}
\def\bop{{\bf p}}
\def\bos{{\bf s}}

\begin{abstract}  We generalize H. Seifert's algorithm for finding a Seifert surface for a knot or link.   The generalization applies to ``framed oriented measured lamination links."   For knots, a Seifert surface determines a unique framing. In our setting, we analyze the set of framed lamination links which bound Seifert laminations and are carried by an $S^1$-fibered tube neighborhood of an oriented train track embedded in a 3-manifold.   \end{abstract}

\section{Introduction.}

The paper \cite{UO:LamLinks} introduced the notion of a ``lamination link."  There is an algorithm due to Herbert Seifert, \cite{HS:SeifertSurface} for finding a Seifert surface for an oriented link.   In this paper, we present a generalization which applies to lamination links.   

To introduce lamination links and their Seifert laminations, we will briefly present an example.  The example introduces some of the main ideas related to lamination links, but the reader may need to refer to \cite{UO:LamLinks} for precise definitions.   There is a similar slightly more interesting example of a lamination link in \cite{UO:LamLinks}.   

\begin{example} \label{IntroEx} Figure \ref{SeifertLam2} shows an oriented branched surface $B$ in $S^3$.  We can assign weights to the sectors of the branched surface satisfying the branch equations at branch curves.  Suppose the weights are $x,y,$ as shown, and the weight vector is $\bov=(x,x+y,y)$ as shown.  If we assign integer weights $x,y$, $B(\bov)$ is an oriented surface whose boundary is an oriented link (with a framing).  Even if the weights are rationally related, up to projective equivalence of weight vectors, $\bdry B(\bov)$ represents an oriented link.  But there are, of course weight vectors satisfying the switch equation whose entries are not rationally related.  For example, we could choose $\bov=(1,1+\sqrt 2,\sqrt 2)$.   Such a  weight vector represents a ``Seifert lamination" whose boundary is an oriented measured lamination link which is not a classical link.  

\begin{figure}[H]
\centering
{\includegraphics{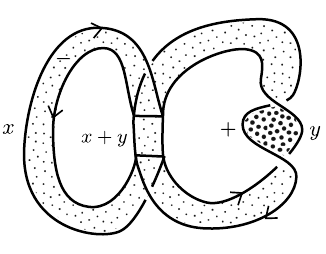}}
\caption{\footnotesize The branched surface $B$ carrying Seifert laminations for framed lamination links.}
\label{SeifertLam2}
\end{figure}

\begin{figure}[H]
\centering
{\includegraphics{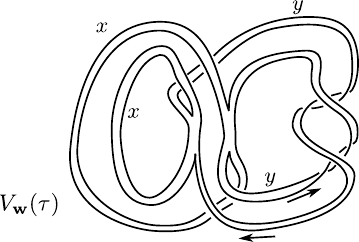}}
\caption{\footnotesize The lamination link.}
\label{LamLink2}
\end{figure}

To picture the laminations, it is best to replace $B$ by a certain fibered neighborhood we call $V(B)$ as shown in Figure \ref{BranchedNeighborhood}.  There is a projection $\pi:V(B)\to B$, and $\pi\inverse(\bdry B)$ is a 2-dimensional train track neighborhood which we will call $V(\tau)$, where $\tau=\bdry B\embed S^3$.  (In this paper the symbol $\embed$ means ``embedded in.")   $V(\tau)$ is a ``framing" of the train track $\tau$ in $S^3$, which is determined here by the embedding of $B$, though in general a framing can be chosen at random.  Informally, the framing of $\tau$ shown in Figure \ref{LamLink2} is obtained by fattening the branched surface $B$ to get $V(B)$; then $V(\tau)$ is the ``edge" of the fattened branched surface.  The weights on sectors of $B$ give weights on the train track $\tau$ which satisfy switch equations and yield an invariant weight vector $\bow$ for $\tau$.  We can imagine the knotted lamination as $V(\tau)$ with the width of different parts of $V(\tau)$ given by the weights of an invariant weight vector $\bow$, see the figure.  The framed train track $\tau$ together with the invariant weight vector $\bow$ for $\tau$ determines a ``prelamination" $V_\bow(\tau)$.   This is a singular foliation of $V(\tau)$, see  Figure \ref{BranchedNeighborhood}, with leaves transverse to fibers of $V(\tau)$ and with a transverse measure such that the measure of a fiber corresponding to a point in the interior of a sector of $\tau$ equals the weight $w_k$ assigned by $\bow$ to the sector.  If $s$ is a sector of $\tau$, $\pi\inverse(\intr(s))$ is foliated as a product and $V_\bow(\tau)$ has a singular foliation (containing the product foliations) as shown in Figure \ref{BranchedNeighborhood}, with singularities at the cusps of $\bdry V(\tau)$.  In much the same way the invariant weight vector $\bov$ for $B$ determines a prelamination representing the Seifert lamination which we denote $B(\bov)$.

\begin{figure}[H]
\centering
\scalebox{1}{\includegraphics{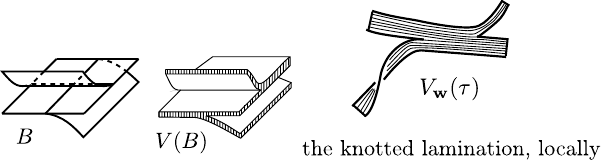}}
\caption{\footnotesize The branched surface $B$, its neighborhood $V(B)$, and the knotted lamination represented as $V_\bow(\tau)$ with weights giving ``widths."}
\label{BranchedNeighborhood}
\end{figure}
\end{example} 

Our goal is to start with a lamination link and produce a Seifert lamination.  As in the classical case, a framed link does not necessarily bound a Seifert lamination.   In fact, only one framing of a classical knot is realized as the boundary of a Seifert surface.   So we must allow for some kind of ``change of framing."

In order to describe our generalized Seifert algorithm, we use a different kind of representation of a lamination link.   Given an invariant weight vector $\bow$ on a train track $\tau\embed S^3$ (see for example Figure \ref{TrainTrack}), we we will describe a large family of lamination links, parametrized by additional real numbers.

\begin{figure}[H]
\centering
\scalebox{0.7}{\includegraphics{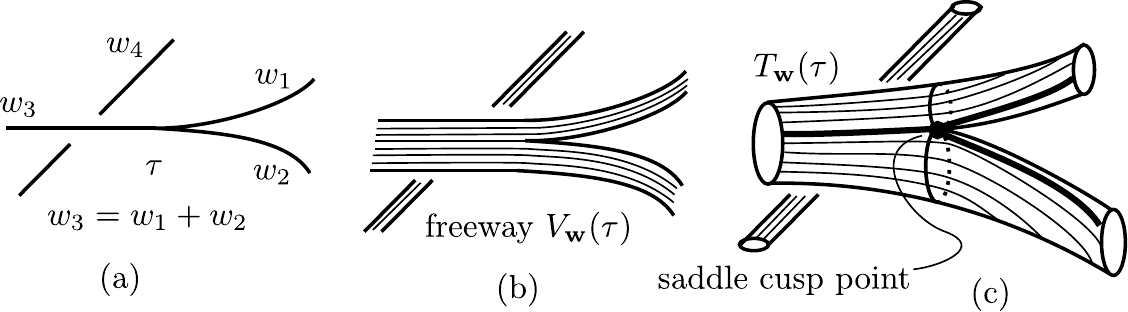}}
\caption{\footnotesize (a) A train track projection in $S^3$ with an invariant vector of weights.  (b) The freeway realization as a lamination link.   (c)  The fibered tube neighborhood realization of the same link.  }
\label{TubeTrack}
\end{figure}

When studying lamination links in $S^3$ or $\reals^3$, we consider a train track with a {\it projection} onto $S^2$ or the plane, as in Figure \ref{TubeTrack}(a).   The train track $\tau$ is divided into {\it sectors}, homeomorphic to intervals or closed curves, by the switch points of $\tau$.  Interval sectors are called {\it segments}. If the train track $\tau$ has an invariant weight vector $\bow$ assigning a weight $w_i$ to the sector $s_i$, then we replace the segment by a foliated product {\it without twisting} when viewed in the projection, where the leaves of the foliation have total transverse measure $w_i$.   The result is the {\it freeway} $V_\bow(\tau)$ associated to the projection, see Figure \ref{TubeTrack}(b).  In the freeway $V_\bow(\tau)$, each sector $s_i$ is represented by the foliated product $s_i\times [0,w_i]$ without half-twists in the projection.  If in $V_\bow(\tau)$ we identify opposite sides of these products, i.e. identify $s_i\times \{0\}$ with $s_i\times \{w_i\}$ for each $i$, we obtain tubes of the form $S^1\times s_i$ foliated by leaves of the form $\{x\}\times s_i$ and the tubes are joined as shown in Figure \ref{TubeTrack}(c) to obtain $T_\bow(\tau)$, representing the same lamination link.  Corresponding to each switch point on $\tau$ there is a {\it saddle cusp point} on $T_\bow(\tau)$.   In the projection we assume that the identified boundaries of $V_\bow(\tau)$ form a train track (the same as $\tau$), always on the top of each tube, as shown, except at saddle cusp points.  If we ignore the measures and the leaves of $T_\bow(\tau)$, $T(\tau)$ is a {\it fibered tube neighborhood} with ``vertical" $S^1$ fibers and a projection $\pi:T(\tau)\to \tau$ which collapses $S^1$ fibers.  For a non-switch point $x$, the corresponding fiber $\pi\inverse(\{x\})$ is a circle, at a switch point $x$ the fiber is a figure eight.   A fibered tube neighborhood is not a tubular neighborhood!  If $N(\tau)$ is a tubular neighborhood we give it the structure of a {\it fibered neighborhood}, with a projection $\pi:N(\tau)\to \tau$.  Here the fiber over a non-switch point is a disk and the fiber over a switch point  is the wedge of two disks at boundary points, such that $\bdry N(\tau)=T(\tau)$, and a fiber ot $T(\tau)$ is the intersection of a fiber of $N(\tau)$ with $T(\tau)$.    A {\it lamination is carried by $T(\tau)$} if it can be embedded transverse to fibers of $T(\tau)$.   It is {\it fully carried} if it intersects every $S^1$ fiber.  We orient $T(\tau)$ such that the outward transverse orientation (pointing out of $N(\tau)$) on $T(\tau)$ is consistent with the orientation of $T(\tau)$.  

Now we observe that there are many other lamination links fully carried by $T(\tau)$ and inducing the same invariant weight vector on $\tau$.   These are obtained by introducing a real-valued Dehn twist (positive or negative) depending on a parameter $t_i\in \reals$ on the tube corresponding to each segment or closed curve $s_i$ of $\tau$, as shown in Figure \ref{Twist}.  The convention for the sign and magnitude are arbitrary, but we will measure the twist by how much a point  on the positive side of the cutting curve is moved relative to the measure on the negative side of the cutting curve.   We make the convention that in the figure the parameter is positive.   We denote the vector of twist parameters $\bot$, and we denote the resulting link $T_{\bow,\bot}(\tau)$.   Once again, we emphasize that our parameters only make sense when we have chosen a projection of $\tau$.  The projection defines the lamination with twist parameters equal to 0, such that $T_{\bow,{\bf 0}}(\tau)=T_\bow(\tau)$.

\begin{figure}[H]
\centering
\scalebox{0.7}{\includegraphics{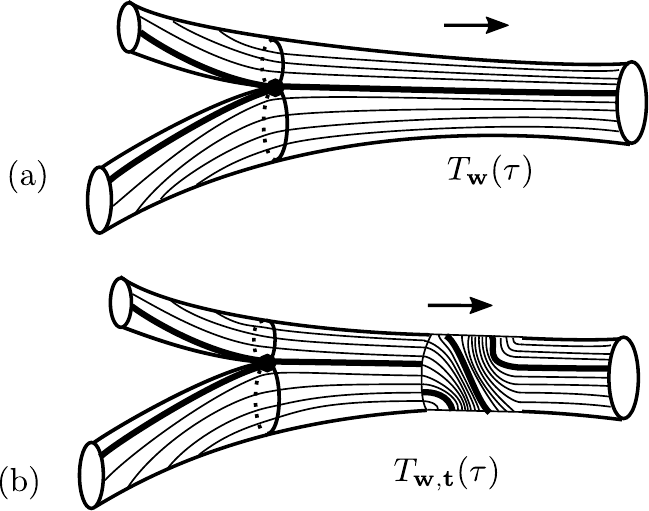}}
\caption{\footnotesize Introducing twists. }
\label{Twist}
\end{figure}

We can change any $T_{\bow,\bot}(\tau)$ to some prelamination $V_\bou(\rho)$ for some train track $\rho$ and invariant weight vector $\bou$ on $\rho$ by splitting on compact arcs in leaves of the singular foliation $T_{\bow,\bot}(\tau)$ emanating from saddle cusp points.

  For every sector of $\tau$, we choose a fiber $\gamma_i$ projecting to a point in the interior of the sector.   We orient $\gamma_i$ such that at the point of $\gamma_i\subset T(\tau)$ the orientation of $\tau$ followed by the orientation of $\gamma_i$ gives the orientation of $T(\tau)$.

The questions we are addressing in this paper are the following.   

\begin{questions}
Given a train track $\tau$ in $S^3$ with an invariant weight vector $\bow$, is it possible to find  framed link $T_{\bow,\bot}(\tau)$ which bounds a Seifert lamination?   If so, is it possible to construct the Seifert lamination explicitly?   Fixing $\bow$, what is the set of $\bot$ such that  $T_{\bow,\bot}(\tau)$ bounds a Seifert lamination in $S^3\setminus N(\tau)$?
\end{questions}

For example, given a train track $\tau$ presented as a projection, with invariant weight vector $\bow$, as shown in Figure \ref{TrainTrack}  (an example from \cite{UO:LamLinks}), can one find a $T_{\bow,\bot}(\tau)$ which bounds a Seifert lamination and construct the Seifert lamination explicitly?  We should point out that in general $T_{\bow,\bot}$ may be equivalent for different values of $(\bow,\bot)$.

\begin{figure}[H]
\centering
\scalebox{0.7}{\includegraphics{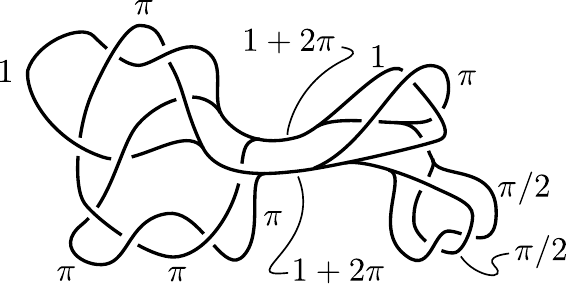}}
\caption{\footnotesize A train track embedded in $S^3$ with an invariant vector of weights.}
\label{TrainTrack}
\end{figure}

\begin{thm}\label{Seifert} [Generalized Seifert algorithm] Suppose $\tau\embed S^3$ is a train track embedded in $S^3$ and $\bow$ is an invariant weight vector  for $\tau$ having all entries positive.   Then there exists a twist parameter vector $\bot$ such that the link  $T_{\bow,\bot}(\tau)$ bounds a Seifert lamination in the complement of the fibered neighborhood $N(\tau)$ of $\tau$.   There is an explicit algorithm to find $\bot$ and the Seifert lamination.   If $\bow$ has all integer entries, the algorithm yields integer entries for $\bot$ and a Seifert surface in the complement of $N(\tau)$.
\end{thm}

There is a hidden subtlety.  Given parameters $(\bow,\bot)$, it is possible that $T_{\bow,\bot}(\tau)$ bounds a Seifert lamination in $S^3$ which intersects $\intr(N(\tau))$ essentially.   This means that one should be using a different train track $\tau$, as one can see in the analysis of \cite{UO:LamLinks}. In proving the above theorem, we construct a Seifert lamination disjoint from $\intr(N(\tau))$, so we are able to avoid the issue. 

Just as in the classical case, there are easy homological reasons for the existence of a Seifert lamination as in Theorem \ref{Seifert}.  To explain this, we introduce a few more definitions.

\begin{defn}
The non-zero invariant weight vectors on any train track $\tau$ form a cone $\C(\tau)$.   Projectivizing by taking a quotient where $\bow\in \C(\tau)$ is equivalent to $\lambda\bow$, we obtain the {\it weight cell}, $\PC(\tau)$.   Often we use $\PC(\tau)$ to denote a particular subspace of $\C(\tau)$, namely $\PC(\tau)=\{\bow\in \C(\tau):\sum_iw_i=1\}$, where $w_i$ are the entries of $\bow$.
\end{defn}

Next we shall consider lamination links in a rational homology sphere $P$.   Given an oriented train track $\tau\embed P$ and given an invariant weight vector $\bow$ for $\tau$, we shall see in Section \ref{Homology} that it is still possible to define twist parameters for $T(\tau)$ relative to an arbitrary choice which defines zero twist.

In the following statement, a {\it oriented surface with meridional boundary} is an oriented surface each of whose boundary components is isotopic to a fiber of $T(\tau)$.  Notice that if $S$ is an oriented surface with more than two meridional boundary  components isotopic to the same fiber $\gamma_i$ but with opposite orientations, then it is always possible to paste two adjacent such boundary curves to get another surface with meridional boundary and with fewer boundary components.    Thus $S$ can be replaced by another oriented surface with meridional boundary with the property that any two boundary components both isotopic to $\gamma_i$ have consistent orientations.

\begin{proposition}\label{HomologySeifert}  Suppose $P$ is a rational homology sphere.  Suppose $\tau \embed P$ is an embedded oriented train track, let $N=N(\tau)$ be a regular neighborhood of $\tau$ and let $M=P\setminus \intr(N)$.  Suppose $h:(M,\bdry M)\to (P,N)$ is the inclusion and $h|_{\bdry M}=j$.  Then, there is a sequence of maps with isomorphisms as shown:

$$\C(\tau)\xrightarrow{g} H_1(\tau;\reals)\xleftarrow[\approx]{\pi_*} H_1(N;\reals))\xleftarrow[\approx]{\bdry}H_2(P,N;\reals))\xleftarrow[\approx]{h_*} H_2(M,\bdry M;\reals))\xrightarrow{\bdry}H_1(\bdry M;\reals))$$ $$\xrightarrow{j_*}H_1(N;\reals))\xrightarrow [\approx]{\pi_*}  H_1(\tau;\reals)) ,$$

\noindent where $g$ is a linear injection which converts an invariant weight vector on $\tau$ to an element of $H_1(\tau)$.   Also $\bdry h_*=j_*\bdry$.   It follows that there exists a twist parameter vector $\bot$ such that the link  $T_{\bow,\bot}(\tau)$ bounds a Seifert lamination.   It also follows that if $S$ is an oriented surface with meridional boundary then $[S]=0$ in $H_2(M,\bdry M;\reals)$ and $[\bdry S]=0$ in $H_1(\bdry M;\reals)$.
\end{proposition}

 The reasoning in the previous proposition is useful for another reason.  For a classical knot in $S^3$, there is a unique framing (twist parameter) given by the boundary of a Seifert surface.   For lamination links, as we will show, in general for a fixed invariant weight vector $\bow$ there will be many different twist parameter vectors $\bot$ yielding links $T_{\bow,\bot}(\tau)$ which bound Seifert laminations.   This is true even for typical lamination knots, which are connected links.  
 
 To describe the set of all lamination links carried by $T(\tau)$ which bound Seifert laminations in $M=P\setminus \intr(N(\tau))$, we need an operation on pairs (or finite sets) of oriented 2-dimensional measured laminations embedded in a 3-manifold $M$, which generalizes ``oriented cut-and-paste" for surfaces.    Given two oriented measured laminations represented as $V_\bou(B)$ and $V_\bov(C)$ carried by oriented branched surfaces $B\embed M$ and $C\embed M$ which are transverse to each other and in general position, we construct a new measured lamination as follows.  Pinch $B\cup C$ along neighborhoods of the intersection train track to obtain a new oriented branched surface $A$, where $B$ and $C$ are immersed in $A$.   Then the invariant weight vectors $\bou$ and $\bov$ yield a new invariant weight vector $\boy$ on $A$ obtained by adding weights where $B$ and $C$ were pinched.  The {\it oriented combination} of $V_\bou(B)$ and $V_\bov(C)$ is $V_\boy(A)$. This process can be iterated to perform oriented cut and paste on more than two laminations.  The result is far from unique;  isotoping $B$ and/or $C$ to a different position yields a different result.   However, if there is a train track $\rho$ in $\bdry M$ such that $\bdry B$ and $\bdry C$ are each embedded in an oriented train track $\rho\embed\bdry M$, then the invariant weight vector induced on $\rho$ is determined.   A {\it weighted oriented surface with weight $r$} is an oriented measured lamination $V_r (S)$, where the weight vector in this case has a single entry $r\in \reals$.   We can, of course take oriented combinations involving weighted oriented surfaces.

 Suppose $S$ is an oriented surface with meridional boundary in $M$.   The orientation of a component of $\bdry S$ may be the same or opposite to the orientation of the fiber of $T(\tau)=\bdry M$.   If a representative circle fiber associated to a sector of $\tau$ is denoted $\gamma_i$, we associate to $S$ a twist parameter $\bou$ whose $i$-th entry is $u_i$, the total number of boundary components of $S$ isotopic to $\gamma_i$ counted with signs according to whether the orientation agrees or disagrees with that of $\gamma_i$.   We let $k$ be the number of sectors of $\tau$, so we have meridional curves $\gamma_i$, $i=1,2,\ldots,k$.  
 
  \begin{thm}\label{AllSeifert}  Suppose $P$ is a rational homology sphere and $\tau \embed S^3$ is an embedded train track with $k$ sectors.   Suppose $T_{\bow,\bot}(\tau)$ bounds a Seifert lamination $V_\bov(B)$ in $M=P\setminus\intr(N(\tau))$.   Then if $S$ is any oriented surface with meridional boundary $\bou$, any oriented combination of $V_r (S)$with $V_\bov(B)$ is a Seifert lamination for $T_{\bow,\bot+r\bou}(\tau)$.  Conversely, if $T_{\bow,\bos}(\tau)$ bounds a Seifert lamination in $M$, then there is a Seifert lamination $V_\boxx(E)$ for $T_{\bow,\bos}(\tau)$ which can be obtained by oriented combination of $V_\bov(B)$ with finitely many weighted oriented surfaces of the form $V_{r\bou}(S)$ with meridional boundary.   It follows that for a fixed $\bow$, the set of $\bot$ such that lamination $T_{\bow,\bot}(\tau)$  bounds a Seifert lamination in $M$ is a hyperplane of some dimension in $\reals^k$. \end{thm}
  
Even for classical links, the above theorem says something:   If $\bow$ has integer entries, then we know by Theorem \ref{Seifert} that there exists $\bot$ with integer entries such that   $T_{\bow,\bot}(\tau)$ bounds a Seifert surface $F$ in $M=P\setminus N(\tau)$.   Suppose the Seifert surface is represented by $F=V_\bov(B)$, where $\bov$ has integer entries.   Then Theorem \ref{AllSeifert} says that if there is a Seifert surface in $M$ for  $T_{\bow,\bos}(\tau)$, $\bos\ne \bot$, then it is  obtained by oriented cut-and-paste of $F$ with oriented surfaces in $M$ having meridional boundary.   Presumably there is not always a {\it minimal genus} Seifert surface for $T_{\bow,\bos}(\tau)$ in $M$.

\section {The algorithm.} \label{algorithm} 

This section describes the generalized Seifert algorithm and the proof of Theorem \ref{Seifert}.
To some extent, we follow the proof in the classical case.

We are given an oriented train track $\tau\embed S^3$ with an invariant weight vector $\bow$.     Without loss of generality, we assume the oriented train track $\tau$ is embedded in $\reals^3$ rather than $S^3$, and we assume $\tau$ is  in general position with respect to a projection to $p:\reals^3\to\reals^2$.  This means crossings are transverse and disjoint from switch points.   
Next we replace $\tau$ by the freeway $V_\bow(\tau)$, which means there are no half-twists when $V_\bow(\tau)$ is viewed in the projection.   See Figure \ref{Freeway}.

\begin{figure}[H]
\centering
\scalebox{0.7}{\includegraphics{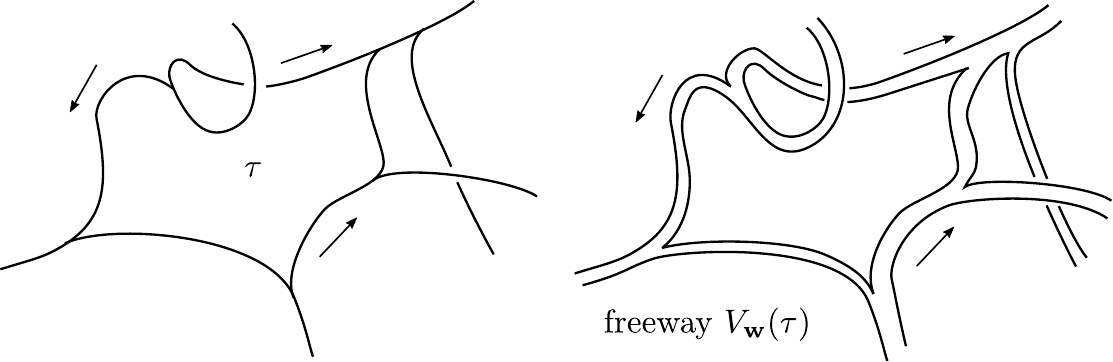}}
\caption{\footnotesize Freeway fibered neighborhood $V_\bow(\tau)$.}
\label{Freeway}
\end{figure}

The first step of the algorithm is to modify $V_\bow(\tau)$ to eliminate crossings.   This is again a measured lamination version of  ``oriented cut-and-paste."  The modification is shown in Figure \ref{CrossingRemove}, and replaces the crossing by a new segment in the train track, assigning a weight equal to the sum of the weights of the segments which cross.  Later in the algorithm, we must remember which segments were obtained by a modification, so we call these segments {\it crossing segments.}  If the two weights at the crossing are equal ($w_1=w_2$ in the figure), one could eliminate the crossing segment by splitting, but we do not.  After eliminating all crossings in this way, we have $V_{\bow'}(\tau')$ embedded in a plane and determining a closed compact oriented measured lamination in the plane.  It follows that the lamination consists of finitely many product families of closed curves.   These families can be obtained by splitting $V_{\bow'}(\tau')$ on compact separatrix leaves joining different cusps.

\begin{figure}[ht]
\centering
\scalebox{0.6}{\includegraphics{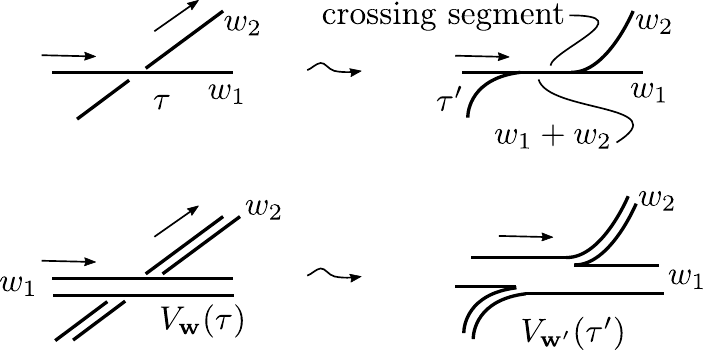}}
\caption{\footnotesize Removing crossings of $V(\tau)$.}
\label{CrossingRemove}
\end{figure}

 We can now construct a 2-dimensionsal Seifert lamination for the modified link without crossings.  The Seifert lamination consists of a union of product families of disks, embedded below the projection plane, with one such product family for each product family of closed curves, such that the boundary of the product of disks gives the product of closed curves.   The product families of closed curves are pinched together in places in $V_{\bow'}(\tau')$, and we can extend the pinching slightly to the product families of disks to obtain a branched surface neighborhood $V_{\bov'}(B')$ satisfying $\bdry V_{\bov'}(B')=V_{\bow'}(\tau')$.  The neighborhood $V_{\bov'}(B')$ is roughly $V_{\bow'}(\tau')$ with all closed curve leaves of capped by disks below the projection plane.
 
In case the reader finds the above description of $V_{\bov'}(B')$ inadequate, we give more details:     Since $V_{\bow'}(\tau')$ has no crossings, we can assume it is embedded in the plane of projection.  Now consider an outermost boundary curve of $V(\tau')$:   It bounds a disk $E$ in the plane of projection with inward boundary cusps.  This implies that the geometric Euler characteristic, $\Chi_g(E)$ is positive.  Hence there is a 0-gon $H$ in the complement $E\setminus \intr(V(\tau'))$.  The orientation on $\tau$ gives a transverse orientation for $\tau$, which orients fibers of $V(\tau)$.   Hence we can find a maximal product $\bdry H\times [0,m]$ in $V_{\bov'}(\tau')$, respecting measures.  Then $\bdry H\times m$ intersects $\bdry_hV(\tau')$ and we can split $V_{\bov'}(\tau')$ on the leaf $\bdry H\times m$.  Removing $\bdry H\times [0,m]$ we obtain $(V(\tau_1'),v_1')$.  Repeating the argument with $(V(\tau_1'),v_1')$, we find a new 0-gon $H$, and split off another product family, to obtain $(V(\tau_2'),v_2')$, and we continue in this way inductively until we have split $V(\tau',v')$ into a finite collection of products of the form $S^1\times [0,m_i]$.   Each of these foliated annuli $S^1\times [0,m_i]$, with transverse measures, bounds a product $D^2\times [0,m_i]$, and the products can be chosen to be disjoint in $\reals^3$, below the plane of projection.   Each product $D^2\times [0,m_i]$ is given an orientation, either a transverse orientation or an orientation on each leaf $D^2\times \{t\}$, such that the induced orientation of $\bdry D^2\times \{t\}$ agrees with the orientation on $\tau$.      Reversing the splittings of $V_{\bov'}(\tau')$ by pinching, and also pinching the disk products in a neighborhood of $V_{\bov'}(\tau')$, we obtain a branched surface neighborhood $V_{\bov'}(B')$ with a measured 2-prelamination obtained from the finite collection of disk products by pinching. 
 
 Our goal now is to modify the Seifert lamination represented by $V_{\bov'}(B')$ to obtain a new Seifert lamination $V_{\bov}(B)$ whose boundary is the original link $V_\bow(\tau)$.   In fact, we cannot do this, and we are forced to settle for something close to the original $V_\bow(\tau)$.    At every crossing segment of $\tau'$, we shall perform an ``inverse modification" which can be extended to the Seifert lamination, but we do not return to the original $V_\bow(\tau)$; rather, we return to a lamination link $V_\bou(\rho)$ which differs by a ``real valued twist" near crossings.   In Figure \ref{Reconstruct}, we show how to modify $V_{\bov'}(B')$ near each crossing segment of $\bdry B' =\tau'$.  Two rectangular regions in $V_{\bow'}(\tau')=\bdry V_{\bov'}(B')$ are shown in Figure \ref{Reconstruct}(a), which we identify as shown.   Doing this at all crossing segments, we obtain  $V_{\bov}(B)$ as shown.  The effect on $V_{\bow'}(\tau')$ is shown in Figure \ref{Reconstruct}(b), yielding $V_\bou(\rho)$.   The figure shows the difference between our original lamination link, represented by $V_\bow(\tau)$ and the new lamination link represented by $V_\bou(\rho)$.   The only difference is a loop introducing a real twist in the lower segment.  Notice that the loops always lie above the lower tape at a crossing, so the loops all lie in a tube neighborhood $T(\tau)$.   There may be more than one loop  over a sector, and the sense of twisting on different loops may not be consistent, so there may be ``cancelling twists."

\begin{figure}[H]
\centering
\scalebox{0.6}{\includegraphics{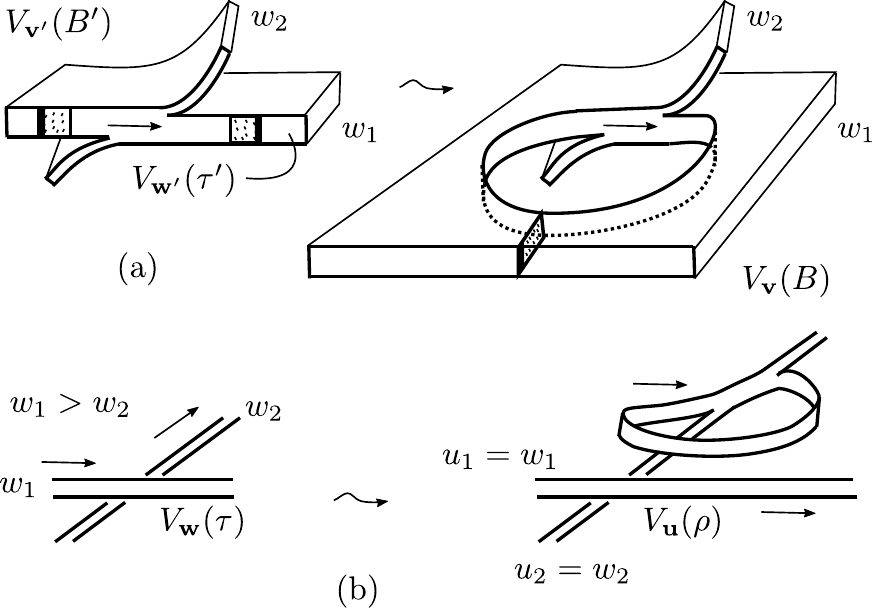}}
\caption{\footnotesize Reconstructing a modified $V_\bow(\tau)$ with a Seifert lamination.}
\label{Reconstruct}
\end{figure}

However, it is always possible to to convert $V_\bou(\rho)$ to a tube neighborhood representative $T_{\bow,\bot}(\tau)$ where the weight vector $\bow$ is unchanged, and the twist parameters, the entries of $\bot$, can be calculated for a particular example.

It remains to consider what happens when all weight entries in $\bow$ are integer weights.  Following through the argument, we see that all 2-laminations involved in the argument are surfaces and the algorithm produces an oriented Seifert surface.

\section {Homological considerations.}\label{Homology}

In this section we explore Seifert laminations via homology.  We always work with a train track $\tau\embed P$ in a rational homology sphere $P$.  Why a rational homology sphere?  Firstly, lamination links carried by $T(\tau)$ represent homology classes with real coefficients in $T(\tau)$, and Seifert laminations represent homology classes with real coefficients in $P\setminus\intr( N(\tau))$.  Secondly, the following very simple example suggests that torsion in $H_1(P;\integers)$ presents no problems for constructing Seifert laminations.

\begin{example}  Suppose $P=\reals P^3$, which is a rational homology sphere.  Let $\tau$ be the union of the 0-cell and the 1-cell in the standard cell decomposition, $\tau$ an oriented simple closed curve.  Then $T(\tau)$ is a torus and the 2-cell in the standard cell decomposition yields a Seifert lamination for $T_{\bow,\bot}(\tau)$ where $\bow=(2)$ and $\bot=(1)$.  To construct the Seifert lamination, replace the 2-cell $E$ by a product $E\times [0,1]\subset M=P\setminus\intr( N(\tau)) $ with $\bdry E\times [0,1]\subset \bdry M$.    This product is a Seifert lamination for $T_{\bow,\bot}(\tau)$

\end{example}

    We assume $\tau$ admits some invariant weight vector $\bow$ with only positive entries.   Let $N=N(\tau)$ denote a regular neighborhood of $\tau$ as described in the introduction, chosen such that $\bdry N= T(\tau)$, and we let $M=P\setminus \intr(N)$.   Finally, we describe a cell structure for $\bdry M=T(\tau)$ which will be crucial for the proof of Proposition  \ref{HomologySeifert}.   We first give a cell structure to $\tau$, using each branch point as a 0-cell and adding one 0-cell on every closed curve component of $\tau$.   This gives the 0-skeleton.    We then add one more ``midpoint" 0-cell in each segment to get a finer cell decomposition of $\tau$.  Note that we do not add a 0-cell to 1-cells in closed curve components of $\tau$

Now we choose an embedding of $e:\tau\to T(\tau)$ such that $\pi\circ e$ is the identity and branch points of $\tau $ coincide with cusp points in $T(\tau)$, see Figure \ref{Cell}.  If $p_i$ is a midpoint of the i-th 1-cell, or the 0-cell in a closed curve component of $\tau$, we attach a {\it meridional} 1-cell  $\gamma_i$ in $T(\tau)$ which is a fiber in $T(\tau)$.   We orient the cells in $e(\tau)$ to agree with the orientation of $\tau$.  We orient $\gamma_i$ such that at the point $e(\tau)\cap \gamma_i$ the orientation of $e(\tau)$ followed by the orientation of $\gamma_i$ gives the orientation of $\bdry N=\bdry M= T(\tau)$.  In the future, we will use $\tau$ to denote $e(\tau)$ when the meaning is clear from the context.
The 1-complex in $T(\tau)$ constructed so far cuts $T(\tau)$ into 2-cells.  When we are working in $\reals^3\subset S^3$, given a projection of $\tau$ to the plane, there is a preferred $e(\tau)$ associated to the projection; each segment of $\tau$ is contained in the ``top half" of $T(\tau)$, as viewed in the projection, see Figure \ref{Cell}.   For $\tau\embed P$, where $P$ is any rational homology sphere, the choice of $\tau\subset T(\tau)$ is arbitrary subject to the condition that $\pi\circ e$ is the identity on $\tau$.   We will abuse notation and sometimes use $\gamma_i$ to denote an oriented closed curve. We suppose there are $k$ sectors and $k$ corresponding oriented fibers $\gamma_i$.

\begin{figure}[H]
\centering
\scalebox{0.7}{\includegraphics{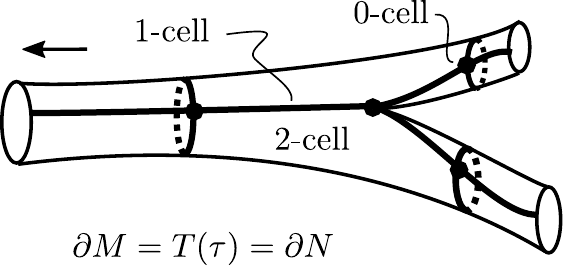}}
\caption{\footnotesize Cell structure for $\bdry M=T(\tau)$. }
\label{Cell}
\end{figure}


 \begin{lemma} \label{K}
 There is a map $K:\C(\tau)\times \reals^k\to H_1(\bdry N)$, $K(\bow,\bot)=[T_{\bow,\bot}]=[\tau(\bow)]+\sum_i t_i[\gamma_i]$ such that $\pi_*\circ K(\bow,\bot)=[\tau(\bow)]\in H_1(\tau,\reals)$.  $K$ has a linearity property $K(\lambda(\bow,\bot))=\lambda K(\bow,\bot)$, $\lambda\ge 0$, and $K$ has the additional property that $K(\bow,-\bot)=2[\tau(\bow)]-K(\bow,\bot)$.   The kernel of $\pi_*\circ K$ is $\{(0,\bot):\sum_i t_i[\gamma_i]=0\}$.   $K(\bow,\bot)\in H_1(\bdry N)$ is the homology class represented by $T_{\bow,\bot}(\tau)$.
 \end{lemma} 
 
The lemma is easily verified.


\begin{proof}[Proof of Proposition \ref{HomologySeifert}]  We have $P$  a rational homology 3-sphere with $\tau\embed P$ an embedded oriented train track.  Let $\bow$ be an invariant weight vector for $\tau$ with positive entries.  As before, we let $M$ denote $P\setminus \intr(N)$, a manifold with boundary, we let $h:(M,\bdry M)\to (P,N)$ be the inclusion, and we use $j$ to denote the inclusion $h|_{\bdry M}=j:\bdry M\to N$. In this proof all homology groups are assumed to have real coefficients.
\hop

We wish to establish the following sequence of maps.

$$\C(\tau)\xrightarrow{g} H_1(\tau)\xleftarrow[\approx]{\pi_*} H_1(N)\xleftarrow[\approx]{\bdry}H_2(P,N)\xleftarrow[\approx]{h_*} H_2(M,\bdry M)\xrightarrow{\bdry}H_1(\bdry M)\xrightarrow{j_*}H_1(N)\xrightarrow [\approx]{\pi_*}  H_1(\tau) ,$$
and show there is a linear injection $g$ which converts an invariant weight vector on $\tau$ to an element of $H_1(\tau)$.  We also want to show that $\bdry h_*=j_*\bdry$.  (Homology with real coefficients.)

\hop

First we define $g:\C(\tau)\to H_1(\tau)$ and show it is injective.   We use the coarser cell decomposition of $\tau$, where the 1-cells are segments of $\tau$, together with one segment in each closed curve component of $\tau$.  We choose orientations for the 1-cells of $\tau$ consistent with the orientation of $\tau$.  Now using the chain complex for cellular homology, an invariant weight vector on $\tau$ immediately gives a cycle in the 1-chains $C_1(\tau)$.   Since there are no 2-cells, no cycle is a boundary, so the 1-cycles $Z_1(\tau)\approx H_1(\tau)$.  The map $g:\C(\tau)\to Z_1(\tau)$ is  defined by $g(\bow)=\sum w_is_i$, where $s_i$ is the $i$-th oriented 1-cell and $w_i$ is the weight on that cell.  (Each 1-cell $s_i$ corresponds to a unique sector, even if it lies in a closed curve sector.)  Cleary $g(\bow)$ is a cycle if $\bow$ is an invariant measure.
  If $g(\bow)=g(\bov)$, then $\sum w_is_i=\sum v_is_i$, so $w_i=v_i$ and $\bow=\bov$.  

Next we define the other maps in the sequence and prove that some of them are isomorphisms. The map $\pi_*:H_1(N)\to H_1(\tau)$ is an isomorphism because $N$ and $\tau$ are homotopy equivalent.  Consider the long exact sequence for relative homology for the pair $(P,N)$.  Since $P$ is a rational homology sphere, $H_2(P)=H_1(P)=0$ , hence $\bdry:H_2(P,N)\to H_1(N)$ is an isomorphism.  By excision, the map $h_*:H_2(M,\bdry M)\to H_2(P,N)$ is an isomorphism.  There is a boundary operator $\bdry:H_2(M,\bdry M)\to H_1(\bdry M)$ in the long exact sequence for the pair $(M,\bdry M)$.   The long exact sequence for $(M,\bdry M)$ is related to the long exact sequence for $(P,N)$ by the inclusion $h:(M,\bdry M)\to (P,N)$ whose restriction to $\bdry M$ is $j$.   Then by naturality, we have $\bdry h_*=j_*\bdry$, whence also $\pi_*\bdry h_*=\pi_*j_*\bdry$ in the sequence above.

Now we shall prove that given an invariant weight vector $\bow$ for $\tau$, there exist twist parameters $\bot$ such $T_{\bow,\bot}$ bounds a Seifert lamination.   Given $\bow$, starting at the left end of the sequence above we get an element in $\alpha\in H_2(M,\bdry M)$, namely $\alpha=h_*\inverse\bdry\inverse\pi_*\inverse g(\bow)$.  We represent $\alpha$ by a 2-dimensional lamination $B(\bov)$ properly embedded in $M$.   This can be done using simplicial homology by desingularizing a cycle representing $\alpha$, just as one represents a class in $H_2(M;\integers)$ by an embedded surface .   However, we want a better understanding of the Seifert lamination representative $B(\bov)$ of $\alpha$.  In particular, we want a better understanding of $\bdry B(\bov)$.   On the one hand, we chose $\alpha$ so that $\pi_*\bdry h_*(\alpha)= [\tau(\bow)]$.   On the other hand, naturality says  $\pi_*\bdry h_*(\alpha)= \pi_*j_*\bdry(\alpha)=[\tau(\bow)]\in H_1(\tau)$.     This means that if we represent $\bdry \alpha$ as a cycle in cellular homology, say $\bdry \alpha =[\sum u_is_i +\sum  t_i\gamma_i]$, then $\pi_*j_*\bdry\alpha=[\sum u_is_i]=[\tau(\bow)]\in H_1(\tau)$, which implies $u_i=w_i$.  The cycle $\sum w_is_i +\sum  t_i\gamma_i$ bounds a relative cycle.   We choose a triangulation of $M$ whose 1-skeleton contains the 1-skeleton of the cell decomposition of $\bdry M$.    Then the cycle  $\sum w_is_i +\sum  t_i\gamma_i$ can be viewed as a cycle in simplicial homology, so it bounds a relative simplicial cycle in $H_2(M,\bdry M)$.   We desingularize this relative cycle to get an oriented measured lamination $B(\bov)$.  By construction, $\bdry B(\bov)=T_{\bow,\bot}(\tau)$.  We have proved the existence of a twist parameter such that the lamination link $T_{\bow,\bot}(\tau)$ bounds a Seifert lamination.

Now we will show that if $S$ is an oriented surface with meridional boundary then $[S]=0$ in $H_2(M,\bdry M;\reals)$ and $[\bdry S]=0$ in $H_1(\bdry M;\reals)$.  We will use naturality again, $\bdry h_*=j_*\bdry$.   Since $[\bdry S]$ is trivial in $H_1(N)$, $j_*\bdry [S]=0$.   Therefore $\bdry h_*[S]=0$.  Because $P$ is a rational homology sphere, the boundary operator $\bdry:H_2(P,N)\to H_1(N)$ is an isomorphism, so $h_*[S]=0$.   Since $h_*$ is an isomorphism by excision, $[S]=0\in H_2(M,\bdry M)$.  Applying $\bdry:H_2(M,\bdry M)\to H_1(\bdry M)$, we see $[\bdry S]=0\in H_1(\bdry M)$.
\end{proof}

  
  \begin{proof}[Proof of Theorem \ref{AllSeifert} ]  One can see, using drawings, the effect of oriented combination of a weighted oriented surface with meridional boundary, $V_r(S)$, with a Seifert lamination $V_\bov(B)$.  Recall that to an an oriented surface $S$ with meridional boundary we can associate a vector $\bou$ which assigns to each segment of $s_i$ of $\tau$ an integer $u_i$ which is the number of meridional boundary components of $S$ which project to a point in $s_i$.   We count $u_i$ with signs according to the orientation of the component of $\bdry S$.   Looking at a positively (negatively) oriented boundary component of $S$, the effect of oriented combination of $\bdry V_r(S)$ with $\bdry V_\bov(B)$ near the boundary component of $S$ is a positive (negative) $r$- twist.   Adding twists associated to all boundary components, we see that the boundary of the combination $V_\boz(C)$ has boundary $T_{\bow,\bot+r\bou}(\tau)$, as required.
 
 Now we prove a converse.  Suppose $T_{\bow,\bot}(\tau)$ bounds a Seifert lamination $V_\bov(B)$ with boundary $T_{\bow,\bot}(\tau)$, and suppose $T_{\bow,\bos}(\tau)$ also bounds a Seifert lamination, say $V_\boxx(A)$.   We want to show that $\bos-\bot$ corresponds to an oriented combination of boundaries of oriented surfaces with meridional boundary.   As an element of $H_1(\bdry M)$, $ [T_{\bow,\bot}(\tau)]=K(\bow,\bot)$ while $[T_{\bow,\bos}(\tau)]=K(\bow,\bos)$, see Lemma \ref{K}.  The Seifert laminations $V_\bov(B)$ and $V_{\boxx}(A)$ represent elements of $H_2(M,\bdry M)$, and $\bdry\left\{[ V_\boxx(A)]-[V_{\bov}(B)]\right\}=[T_{\bow,\bos}(\tau)]- [T_{\bow,\bot}(\tau)]=K(\bow,\bos)-K(\bow,\bot)=K(0,\bos-\bot)$ by Lemma \ref{K}.   This class is represented by the cycle $\sum (s_i-t_i)\gamma_i$, and we know it is the boundary of a relative cycle representing an element of $H_2(M,\bdry M)$.  Just as in the proof of Proposition \ref{HomologySeifert}, we can realize this relative cycle as a lamination with meridional boundary, say $V_\boy(E)$, where $E$ is an oriented branched surface, each of whose boundary components is a meridian.    (Recall, we use cellular homology in $\bdry M$ with our carefully chosen cell decomposition, then use simplicial homology in $M$ with a suitable fine triangulation whose 1-skeleton contains the 1-skeleton of the cell decomposition of $\bdry M$, and finally we desingularize a relative simplicial cycle.) We can write $\boy\in \C(E)$ as a convex combination of rational invariant weight vectors $\boy_i$ for $E$, $\boy=\sum_i t_i\boy_i$.  Then we choose a positive integer $q$ so that $V_{q\boy_i}(E)$ is a surface $S_i$ for all $i$.  $V_\boy(E)$ is then obtained as an oriented combination of the weighted surfaces $V_{r_i}(S_i)$ for suitable real weights $r_i$.   Hence also the oriented combination of $V_{r_i}(S_i)$, $i=1,2,\ldots, \ell$ and $V_\bov(B)$ yields a lamination with the same boundary as $V_\boxx(A)$.
 \end{proof}

\bibliographystyle{amsplain}
\bibliography{ReferencesUO3}
\end{document}